\documentclass[11pt,reqno]{amsart}
\usepackage{amscd,amssymb,psfrag,amsxtra}
\usepackage[dvipdfm]{graphicx}
\usepackage{bmpsize}
\usepackage{caption}
\usepackage{subcaption}
\usepackage{float}
\usepackage{pinlabel}

\newcommand{\p}{\partial}
\newcommand{\Z}{\mathbb Z}

\newcommand{\C}{\mathbb C}

\newcommand{\R}{\mathcal R}

\newcommand{\im}{\operatorname{im}}

\newcommand{\ep}{\varepsilon}
\renewcommand{\phi}{\varphi}

\newcommand{\tr}{\operatorname{tr}}

\newcommand{\cs}{\mathbf{cs}}

\newcommand{\Ad}{\operatorname{Ad}}

\newcommand{\Hom}{\operatorname{Hom}}

\newcommand{\Tor}{\operatorname{Tor}}
\newcommand{\Aut}{\operatorname{Aut}}

\newcommand{\diag}{\operatorname{diag}}

\newcommand{\Int}{\operatorname{Int}}
\newcommand{\PD}{\operatorname{PD}}

\newtheorem{theorem}{Theorem}[section]
\newtheorem{lemma}[theorem]{Lemma}
\newtheorem{proposition}[theorem]{Proposition}
\newtheorem{corollary}[theorem]{Corollary}

\newtheorem*{thm}{Theorem}

\theoremstyle{definition}
\newtheorem*{remark}{Remark}
\newtheorem*{example}{Example}

\def\acknowledgementname{Acknowledgements.}

\begin{document}

\title{On the deleted squares of lens spaces}
\thanks{Both authors were partially supported by NSF Grant 1065905.}
\author{Kyle Evans-Lee}
\address{Department of Mathematics, University of Miami, Coral Gables, FL 33124}
\email{\rm{kyleevanslee@gmail.com}}
\author{Nikolai Saveliev}
\address{Department of Mathematics, University of Miami, Coral Gables, FL 33124}
\email{\rm{saveliev@math.miami.edu}}


\subjclass[2000]{55R80, 55S30, 57M27, 57R19}

\begin{abstract}
The configuration space $F_2 (M)$ of ordered pairs of distinct points in a manifold $M$, also known as the deleted square of $M$, is not a homotopy invariant of $M$: Longoni and Salvatore produced examples of homotopy equivalent lens spaces $M$ and $N$ of dimension three for which $F_2 (M)$ and $F_2 (N)$ are not homotopy equivalent. In this paper, we study the natural question whether two arbitrary $3$-dimensional lens spaces $M$ and $N$ must be homeomorphic in order for $F_2 (M)$ and $F_2 (N)$ to be homotopy equivalent. Among our tools are the Cheeger--Simons differential characters of deleted squares and the Massey products of their universal covers.
\end{abstract}

\maketitle

\section{Introduction}
The configuration space $F_n (M)$ of ordered $n$-tuples of pairwise distinct points in a manifold $M$ is a much studied classic object in topology. Until a few years ago, it was conjectured that homotopy equivalent manifolds $M$ must have homotopy equivalent configuration spaces $F_n (M)$. Much had been done towards proving this conjecture until in 2004 Longoni and Salvatore \cite{SL} found a counterexample using the non-homeomorphic but homotopy equivalent lens spaces $L(7,1)$ and $L(7,2)$. They proved that the configuration spaces $F_2 (L(7,1))$ and $F_2 (L(7,2))$ are not homotopy equivalent by showing that their universal covers have different Massey products: all of the Massey products vanish for the former but not for the latter.

This result prompted a natural question pertaining specifically to lens spaces: does the homotopy type of configuration spaces distinguish all lens spaces up to homeomorphism\,? This question was studied by Miller \cite{miller} for two--point configuration spaces, also known under the name of deleted squares. Miller extended the Massey product calculation of \cite{SL} to arbitrary lens spaces; however, comparing the resulting Massey products turned out to be too difficult and the results proved to be inconclusive.

In this paper, we study the above question using Cheeger--Simons flat differential characters \cite{CS}. With their help, we obtain new algebraic restrictions on possible homotopy equivalences between deleted squares of lens spaces. In particular, these restrictions allow for a much easier comparison of the Massey products, producing multiple examples of pairs of homotopy equivalent lens spaces whose deleted squares are not homotopy equivalent, see Section \ref{S:natural}. Some of these pairs have non-vanishing sets of Massey products on both of their universal covers. It remains to be seen if our techniques are sufficient to answer the question in general.

Here is an outline of the paper. The bulk of it deals with Cheeger--Simons flat differential characters. Given a lens space $L(p,q)$, denote by $X_0$ its configuration space $F_2 (L(p,q))$. We study the Cheeger--Simons character $\cs$ which assigns to each representation $\alpha: \pi_1 (X_0) \to SU(2)$ a homomorphism $\cs(\alpha): H_3 (X_0) \to \mathbb R/\Z$. This homomorphism is obtained by realizing each of the generators of $H_3 (X_0)$ by a continuous map $f: M \to X_0$ of a closed oriented $3$--manifold $M$ and letting $\cs(\alpha)$ of that generator equal the Chern--Simons function of the pull-back representation $f^*\alpha$. 

The realization problem at hand is known to have a solution due to an abstract isomorphism $\Omega_3 (X_0) = H_3 (X_0)$, however, explicit realizations $f: M \to X_0$ have to be constructed by hand. Using the naturality of $\cs$, we reduce this task to a somewhat easier problem of realizing homology classes in $H_3 (L(p,q) \times L(p,q))$ and solve it by finding a set of generators realized by Seifert fibered manifolds. The Chern--Simons theory on such manifolds is sufficiently well developed for us to be able to finish the calculation of $\cs$. This calculation leads to the following theorem, which constitutes the main technical result of the paper.

\begin{thm}
Let $p$ be an odd prime and assume that the deleted squares $X_0$ and $X'_0$ of lens spaces $L(p,q)$ and $L(p,q')$ are homotopy equivalent. Then there exists a homotopy equivalence $f: X'_0 \to X_0$ such that, with respect to the canonical generators of the fundamental groups, the homomorphism $f_*: \pi_1 (X'_0) \to \pi_1 (X_0)$  is given by a scalar matrix $\diag\,(\alpha,\alpha)$, where $\pm\,q' = q\alpha^2\pmod p$.
\end{thm}

The homotopy equivalence $f: X'_0 \to X_0$ of this theorem lifts to a homotopy equivalence $\tilde f: \tilde X'_0 \to \tilde X_0$ of the universal covering spaces of the type studied by Longoni and Salvatore \cite{SL} and Miller \cite{miller}. The homotopy equivalence $\tilde f$ naturally possesses equivariance properties made explicit by the knowledge of the induced map $f_*$ on the fundamental groups. In the rest of the paper, we use these properties to reduce the comparison problem for the Massey products on $\tilde X'_0$ and $\tilde X_0$ to an algebraic problem in certain cyclotomic rings arising as cohomology of $\tilde X_0$ and $\tilde X'_0$. This is still a difficult problem, which is solved in individual examples with the help of a computer.

\medskip\noindent
\textbf{Acknowledgments:} We are thankful to S{\l}awomir Kwasik who brought the problem discussed in this paper to our attention and who generously shared his expertise with us. We also thank Ken Baker for his invaluable help throughout the project, Ian Hambleton, Matthew Miller, Daniel Ruberman, Paolo Salvatore, and Dev Sinha for useful discussions, and Joe Masterjohn for assisting us with the computer code.


\section{Homology calculations}
This section covers some basic homological calculations for lens spaces, their squares, and their deleted squares which are used later in the paper.


\subsection{Lens spaces}\label{S:lenses} 
Let $p$ and $q$ be relatively prime positive integers such that $p > q$. Define the lens space $L(p,q)$ as the orbit space of the unit sphere $S^3 = \{\,(z,w)\;|\;|z|^2 + |w|^2\,\} \subset \C^2$ by the action of the cyclic group $\Z/p$ generated by the rotation $\rho(z,w) = (\zeta z,\zeta^q w)$, where $\zeta = e^{2\pi i/p}$. This choice of $\rho$ gives a canonical generator $1 \in \pi_1 (L(p,q)) = \Z/p$. The standard CW-complex structure on $L(p,q)$ consists of one cell $e_k$ of dimension $k$ for each $k = 0,1,2,3$. The resulting cellular chain complex 
\smallskip
\[
\begin{CD}
0 @>>> \Z @> 0>> \Z @> p>> \Z @> 0>> \Z @>>> 0
\end{CD}
\]

\medskip\noindent
has homology $H_0(L(p,q))= \Z$, $H_1(L(p,q))= \Z/p$, $H_2(L(p,q))= 0$ and $H_3(L(p,q))= \Z$.


\subsection{Squares of lens spaces}\label{S:squares}
Given a lens space $L = L(p,q)$, consider its square $X = L\times L$ and give it the product CW-complex structure with the cells $e_i \times e_j$. The homology of $X$ can easily be calculated using this CW-complex structure\,:
\begin{alignat*}{1}
& H_0 (X) = \Z,\quad H_1 (X) = \Z/p\,\oplus\,\Z/p,\quad H_2 (X) = \Z/p,  \\
& H_3 (X) = \Z\,\oplus\,\Z\,\oplus\,\Z/p,\quad H_4 (X) = \Z/p\,\oplus\,\Z/p,  \\
& H_5 (X) = 0,\quad H_6 (X) = \Z.
\end{alignat*}
This calculation also provides us with explicit generators in each of the above homology groups. For instance, the two infinite cyclic summands in $H_3 (X)$ are generated by $e_0\times e_3$ and $e_3 \times e_0$ and are realized geometrically by the two factors of $L$ in $L \times L$. The summand $\Z/p$ is generated by the homology class of the cycle $e_2\times e_1 + e_1\times e_2$, see Hatcher \cite[page 272]{hatcher}. It is a cycle because $\p (e_2 \times e_1 + e_1 \times e_2)\, =\, p\,\cdot\, e_1\times e_1 - e_1 \times\, p\,\cdot\, e_1 = 0$, and its $p$-th multiple is a boundary $\p (e_2\, \times\, e_2)\; =\; p\,\cdot\, e_1\,\times\, e_2 + e_2\,\times\, p\,\cdot\, e_1 = p\;(e_2 \times e_1 + e_1 \times e_2)$. A non-singular geometric realization of this homology class will be constructed in Section \ref{S:sing}.


\subsection{Deleted squares}
Let $\Delta \subset X = L \times L$ be the diagonal and call $X_0 = X - \Delta$ the \emph{deleted square} of $L$. It is an open manifold, which contains as a deformation retract the compact manifold $X\,-\,\Int N(\Delta)$, where $N(\Delta)$ is a tubular neighborhood of $\Delta \subset X$. The normal bundle of $\Delta$ is trivial because it is isomorphic to the tangent bundle of $L$ and the manifold $L$ is parallelizable. Therefore, the boundary of $X - \Int N(\Delta)$ is homeomorphic to $L \times S^2$.

We wish to compute (co)homology of deleted squares. To this end, consider the homology long exact sequence of $(X,\Delta)$. The maps $i_*: H_k (\Delta) \to H_k (X)$ induced by the inclusion of the diagonal are necessarily injective, hence the long exact sequence splits into a family of short exact sequences,
\[
\begin{CD}
0 @>>> H_k (\Delta) @>>>  H_k (X) @>>>  H_k (X,\Delta) @>>> 0. \\
\end{CD}
\]
Calculating $H_k (X,\Delta)$ from these exact sequences is immediate except for $k = 3$. By composing the inclusion $\Delta \subset L \times L$ with the projection on the two factors, we see that the map $i_*: H_3 (\Delta) \to H_3 (X)$ is of the form $i_*(1) = (1,1,a)$ with respect to the generators of $H_3 (X) = \Z\,\oplus\,\Z\,\oplus\,\Z/p$ described in Section \ref{S:squares}. Therefore, $H_3 (X,\Delta) = \Z\,\oplus\,\Z/p$. Using the Poincar\'e duality isomorphism $H^{6-k}(X - \Delta) = H_k (X,\Delta)$ we then conclude that
\begin{alignat*}{1}
& H^0 (X_0) = \Z,\quad H^1 (X_0) = 0,\quad H^2 (X_0) = \Z/p\,\oplus\,\Z/p  \\
& H^3 (X_0) = \Z\,\oplus\,\Z/p,\quad H^4 (X_0) = \Z/p,\quad H^5 (X_0) = \Z/p.
\end{alignat*}

To compute homology of $X_0$, one can repeat the above argument starting with the cohomology long exact sequence of $(X,\Delta)$, or simply use the universal coefficient theorem. The answer is as follows\,:
\begin{alignat*}{1}
& H_0 (X_0) = \Z,\quad H_1 (X_0) = \Z/p\,\oplus\,\Z/p,\quad H_2 (X_0) = \Z/p, \\
& H_3 (X_0) = \Z\,\oplus\,\Z/p,\quad H_4 (X_0) = \Z/p,\quad H_5 (X_0) = 0.
\end{alignat*}

\begin{lemma}\label{L:torsion}
The homomorphism $H_3 (X_0) \to H_3 (X)$ induced by the inclusion $i: X_0 \to X$ is an isomorphism $\Z/p \to \Z/p$ on the torsion subgroups.
\end{lemma}

\begin{proof}
Let $i:X_0 \to X$ be the inclusion map. We will show that $i_*: H_3(X_0)\to H_3(X)$ is injective, which will imply the result because $H_3(X_0)=\Z\,\oplus\,\Z/p$, $H_3(X)=\Z\,\oplus\,\Z\,\oplus\,\Z/p$, and all homomorphisms $\Z/p\to \Z$ are necessarily zero.

Apply the excision theorem to the pair $(X,X_0)$ with $U=X - N(\Delta)$ to obtain $H_*(X,X_0)= H_* (X-U,X_0 - U)= H_* (N(\Delta),\partial N(\Delta))$. Using the Thom isomorphism $H_* (N(\Delta),\partial N(\Delta))= H_{*-3} (\Delta)$, we conclude that $H_2 (X,X_0) = 0$, $H_3 (X,X_0) = \Z$ and $H_4 (X,X_0)=\Z/p$. In particular, the long exact sequence 
\medskip
\[
\begin{CD}
H_4(X_0) @> i_*>> H_4(X) @> j >> H_4(X,X_0) @>\delta>> H_3(X_0)@>i_*>> H_3(X) \\
\end{CD}
\]

\medskip\noindent
of the pair $(X,X_0)$ looks as follows
\medskip
\[
\begin{CD}
\Z/p @ >i_*>>\Z/p \oplus \Z/p @> j >> \Z/p @>\delta>>\Z\oplus\Z/p @> i_*>> \Z \oplus\Z\oplus\Z/p 
\end{CD}
\]

\medskip\noindent
We wish to show that $\delta=0$. Since there are no non-trivial homomorphisms $\Z/p \to \Z$, the map $\delta$ must send a generator of $ \Z/p$ to $(0,k)\in \Z\,\oplus\,\Z/p$ for some $k\pmod p$. Then $\ker \delta = \im j = \Z/m$, where $m=\gcd(p,k)$. By the first isomorphism theorem applied to $j: \Z/p \oplus \Z/p\to\Z/p$ with $\ker j = \im i_*$ we have 
\[
(\Z/p\,\oplus\,\Z/p)/\im i_*\, =\, \Z/m,
\]
which is only possible if $m = p$. But then $\gcd(p,k) = p$ which implies that $k = 0\pmod p$ and hence $\delta = 0$.
\end{proof}

Recall that we have a canonical isomorphism $H_3 (X) = \Z\,\oplus\,\Z\,\oplus\,\Z/p$ given by the choice of generators $e_0 \times e_3$, $e_3 \times e_0$ and $e_2 \times e_1 + e_1 \times e_2$.
  
\begin{lemma}\label{L:gens}
One can choose generators in $H_3 (X_0) = \Z\,\oplus\,\Z/p$ so that the homomorphism $i_*: H_3 (X_0) \to H_3 (X)$ sends $(1,0) \in H_3 (X_0)$ to $(1,1,0) \in H_3 (X)$, and  $(0,1) \in H_3 (X_0)$ to $(0,0,1) \in H_3 (X)$.
\end{lemma}

\begin {proof}
It follows from the proof of Lemma \ref{L:torsion} that the homomorphism $\delta: H_4 (X,X_0) \to H_3 (X_0)$ in the homology long exact sequence of the pair $(X,X_0)$ is zero. Therefore, that sequence takes the form
\medskip
\[
\begin{CD}
0\to H_3(X_0)@>i_*>> H_3(X) \to H_3(X,X_0) \to H_2(X_0) @> i_* >> H_2(X) \to 0. \\
\end{CD}
\]

\medskip\noindent
Since both $H_2 (X_0)$ and $H_2 (X)$ are isomorphic to $\Z/p$, the homomorphism $i_*: H_2 (X_0) \to H_2 (X)$ must be an isomorphism. According to Lemma \ref{L:torsion}, the homomorphism $i_*: H_3 (X_0) \to H_3 (X)$ is an isomorphism on the torsion subgroups. Factoring out the torsion, we obtain the short exact sequence 
\medskip
\[
\begin{CD}
0@>>> H_3(X_0)/\Tor@>i_*>> H_3(X)/\Tor @>>> \Z @>>> 0 \\
\end{CD}
\]


\medskip\noindent
To describe the image of $i_*$ in this sequence, consider the projection of $X = L\,\times\,L$ onto its first factor. The restriction of this projection to $X_0 \subset X$ is a fiber bundle $\pi: X_0 \to L$ with fiber a punctured lens space. Since $L$ is parallelizable, $\pi$ admits a section $s: L \to X_0$ obtained by pushing the diagonal $\Delta \subset X$ off in the direction of the fiber. The group $H_3(X_0)/\Tor$ is then generated by $s_* ([L])$ which is sent by $i_*$ to $[\Delta] = (1,1)$. The statement now follows.
\end {proof}


\section{Geometric realization}\label{S:sing}
In this section, we will realize the homology class $[e_2\times e_1 + e_1\times e_2] \in H_3 (X)$ geometrically by constructing a closed oriented 3-manifold $M$ and a continuous map $f: M \to X$ such that 
\[
f_*\,[M]\;=\;[e_2\times e_1 + e_1\times e_2] \in H_3 (X),
\]
where $[M]$ is the fundamental class of $M$. Finding a pair $(M,f)$ like that is a special case of the Steenrod realization problem. In the situation at hand, this problems is known to have a solution  \cite{rudyak}, which is unfortunately not constructive. Exhibiting an explicit $(M,f)$ will therefore be our task.


\subsection{Singular representative}\label{S:sing-rep}
Let us consider the CW-subcomplex $Y \subset X$ obtained from the 3-skeleton of $X = L \times L$ by removing the cells $e_0\times e_3$ and $e_3\times e_0$. Note that $Y$ contains two cells, $e_1\times e_2$ and $e_2\times e_1$, in dimension three, and that all other cells of $Y$ have lower dimensions. One can easily see that the cycle $e_1\times e_2 + e_2 \times e_1$ generates $H_3 (Y) = \Z$ and that the map $H_3 (Y) \to H_3 (X)$ induced by the inclusion takes this generator to $[e_1\times e_2 + e_2 \times e_1] \in H_3 (X)$. 

The CW-complex $Y$ is easy to describe: it is obtained by attaching solid tori $D^2 \times S^1$ and $S^1 \times D^2$ to the core torus $S^1\times S^1$ via the maps 
\begin{gather*}
\p D^2 \times S^1 \to S^1 \times S^1,\quad (z,w) \to (z^p, w),\\
S^1 \times \p D^2 \to S^1 \times S^1,\quad (z,w) \to (z, w^p),
\end{gather*}
where we identified $D^2$ with the unit disk in complex plane. Note that $Y$ is not a manifold unless $p = 1$ (in which case it is the 3-sphere with its standard Heegaard splitting). Our next step will be to find, for every $p \ge 2$, a closed oriented 3-manifold $M$ and a continuous map $g: M \to Y$ which induces an isomorphism $g_*: H_3 (M) \to H_3 (Y)$. The desired map $f: M \to X$ will then be the composition of $g$ with the inclusion $i: Y \to X$.


\subsection{Resolution of singularities}
The manifold $M$ will be obtained by resolving the singularities of $Y = (D^2 \times S^1) \cup_{\,S^1 \times S^1} (S^1 \times D^2)$. We start by removing tubular neighborhoods of $\{0\} \times S^1\subset D^2\times S^1$ and  $ S^1 \times \{0\} \subset S^1\times D^2$ from the two solid tori to obtain
\[
Y_0=(A \times S^1)\; \cup_{\,S^1\times S^1}(S^1\times A),
\]
where $A$ is an annulus depicted in Figure \ref{fig1}. We will resolve the singularities of $Y_0$ and then fill in the two solid tori to obtain $M$.

\begin{figure}[!ht]
 \begin{minipage}[b]{0.45\linewidth}
  \hspace{-10mm}
  \includegraphics[width=1.45\textwidth]{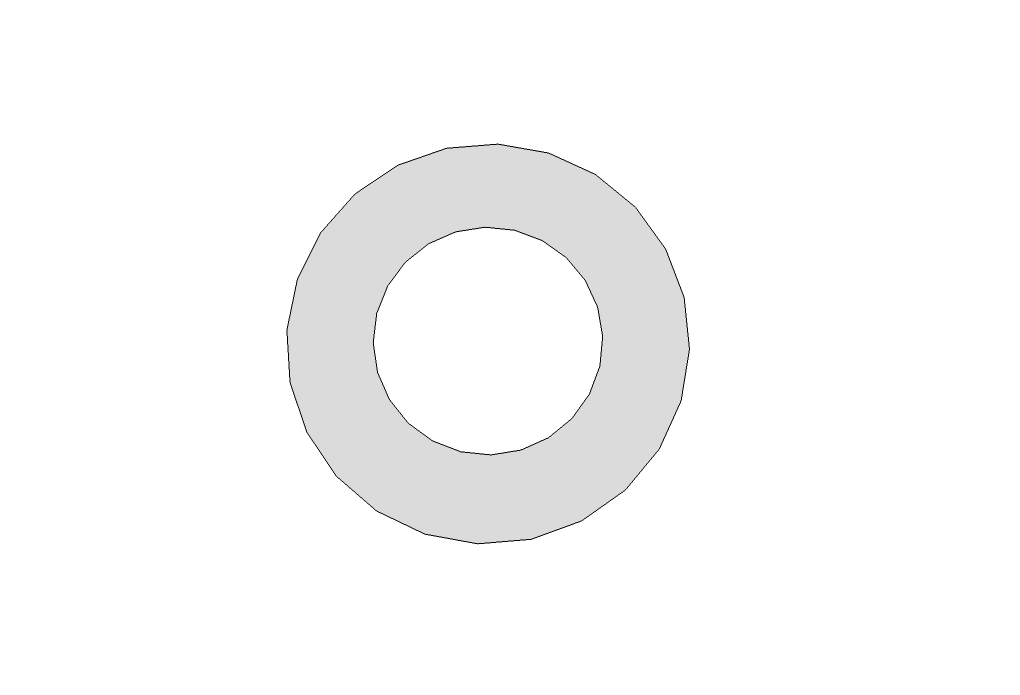}
  \caption{Annulus $A$}\label{fig1}
 \end{minipage}
 \begin{minipage}[b]{0.54\linewidth}
  \includegraphics[width=1.2\textwidth]{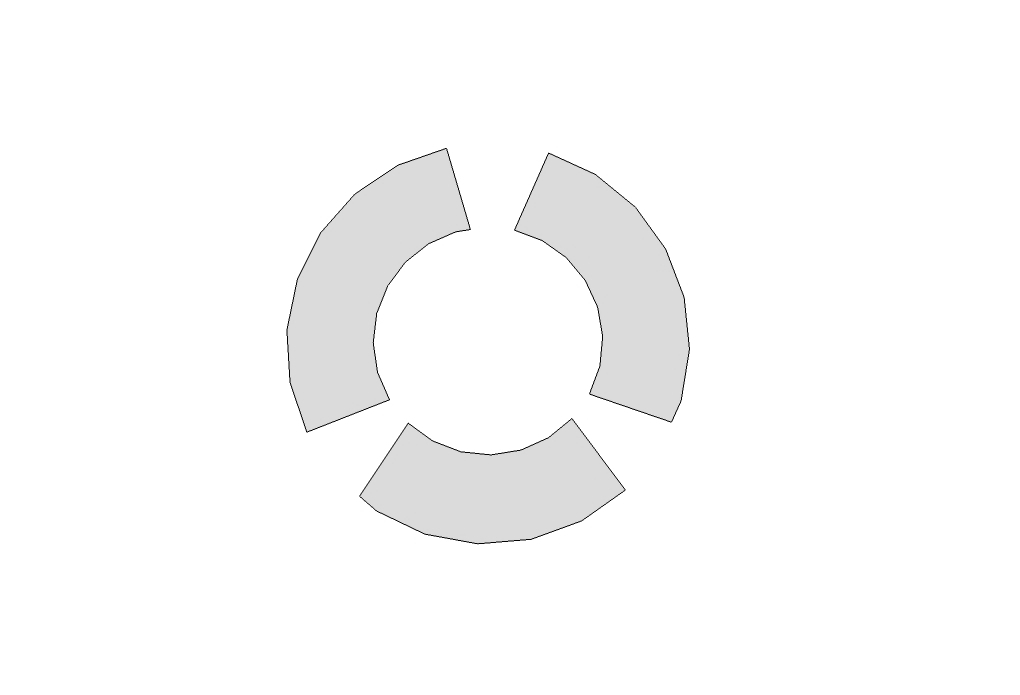} 
  \caption{Cuts in $A$ ($p=3$)}\label{fig2}
 \end{minipage}
\end{figure}

Inside of $Y_0$, one of the two boundary circles of $A$, say the inner one, is glued to an $S^1$ factor in the torus $S^1 \times S^1$ by the map $z \to z^p$ of degree $p$, forming a surface which is singular (unless $p = 2$, when the surface is the M{\"o}bius band). The boundary of this singular surface is a circle. Make $p$ cuts in $A$ as shown in Figure \ref{fig2}. For each of the resulting rectangles $A_i$, $i = 0,\ldots, p-1$, identify the endpoints of its inner side with each other to obtain a circle $C_i$. The resulting spaces $\bar{A_i}$ are shown in Figure \ref{fig3}.

\begin{figure}[!ht]
 \begin{minipage}[b]{0.48\linewidth}
  \hspace{-2mm}
  \includegraphics[width=1.1\textwidth]{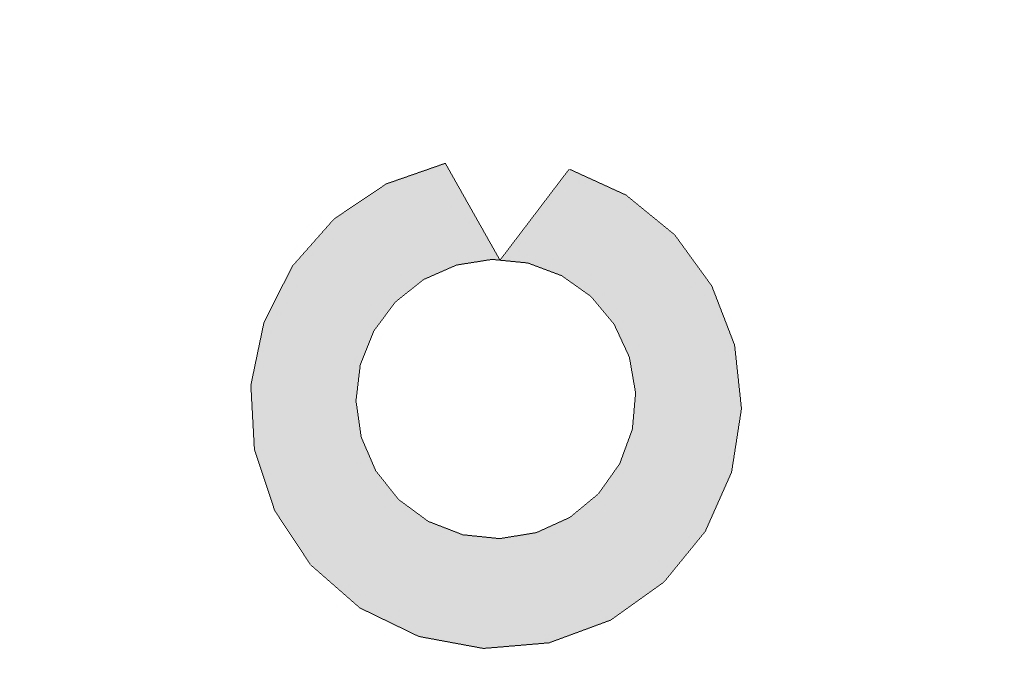}
  \caption{$\bar A_i$}\label{fig3}
 \end{minipage}
 \begin{minipage}[b]{0.48\linewidth}
  \includegraphics[width=\textwidth]{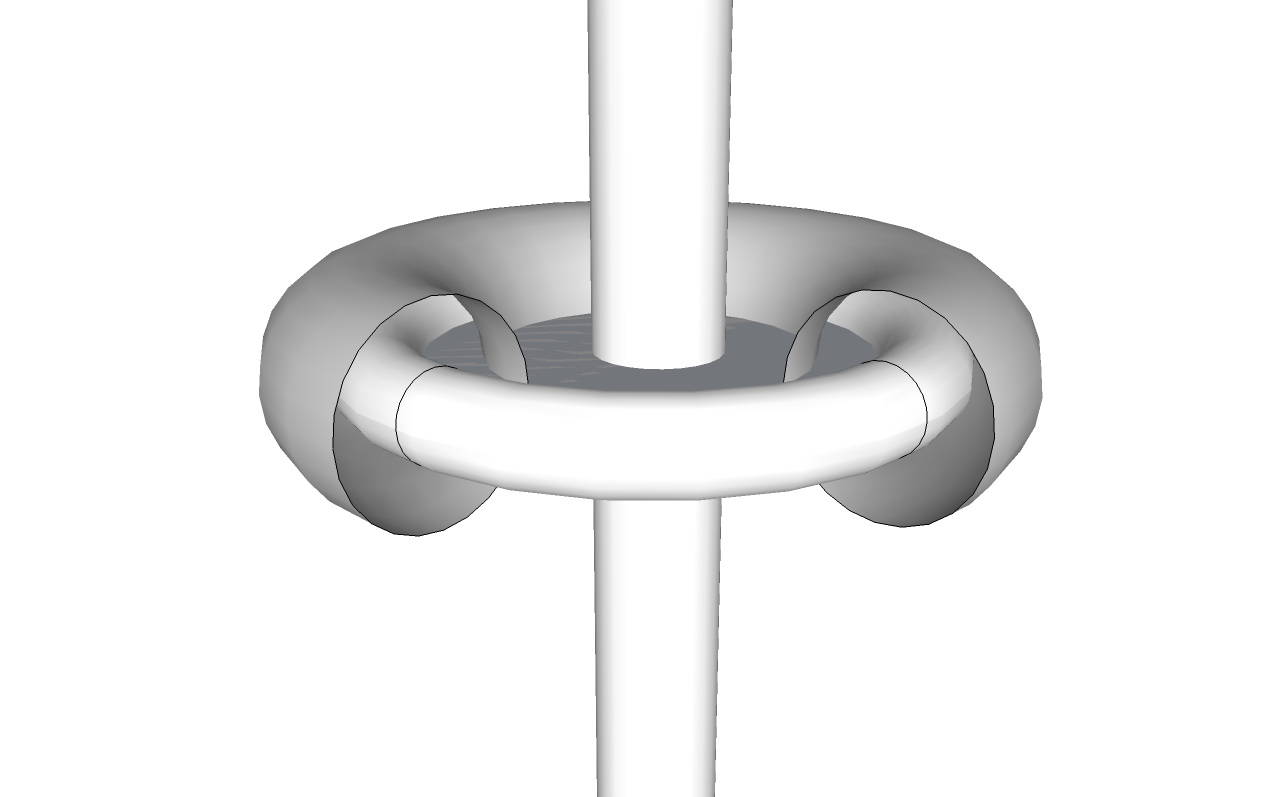} 
  \caption{$W_i$}\label{fig4}
 \end{minipage}
\end{figure}

Next, for each $\bar{A_i} \times S^1$, take a copy of $S^1 \times \bar{A_i}$ coming from the other side, and glue the two together into
\[
W_i = (S^1 \times \bar A_i)\; \cup_{\,S^1\times S^1} (\bar A_i \times S^1)
\]
by identifying the two tori, $C_i\,\times\, S^1$ and $S^1\, \times\, C_i$, via matching the factor $C_i$ of the former with the factor $S^1$ of the latter, and vice versa. The resulting space $W_i$ is a copy of $S^3$ equipped with the standard genus-one Heegaard splitting, from which the cores of the two solid tori have been removed. Additionally, there are two slits cut along the annuli that run along the solid tori as shown in Figure \ref{fig4}. Note that the only singularity in $W_i$ is the point where the two slits meet. Shown in Figure \ref{fig5} is the complement of $W_i$ in the $3$-sphere, the two slits depicted as 2-spheres.

\begin{figure}[!ht]
 \begin{minipage}[b]{0.48\linewidth}
  \hspace{-17mm}
  \includegraphics[width=1.5\textwidth]{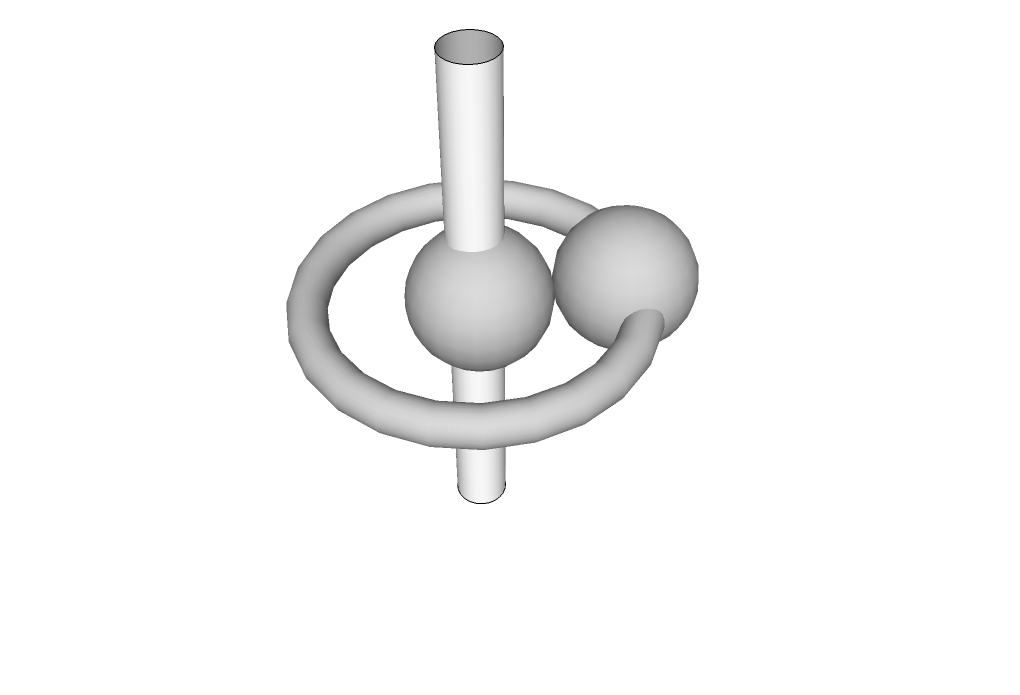}
  \vspace{-15mm}
  \caption{The complement of $W_i$ in $S^3$}\label{fig5}
 \end{minipage}
 \begin{minipage}[b]{0.48\linewidth}
  \hspace{-15mm}
  \includegraphics[width=1.7\textwidth]{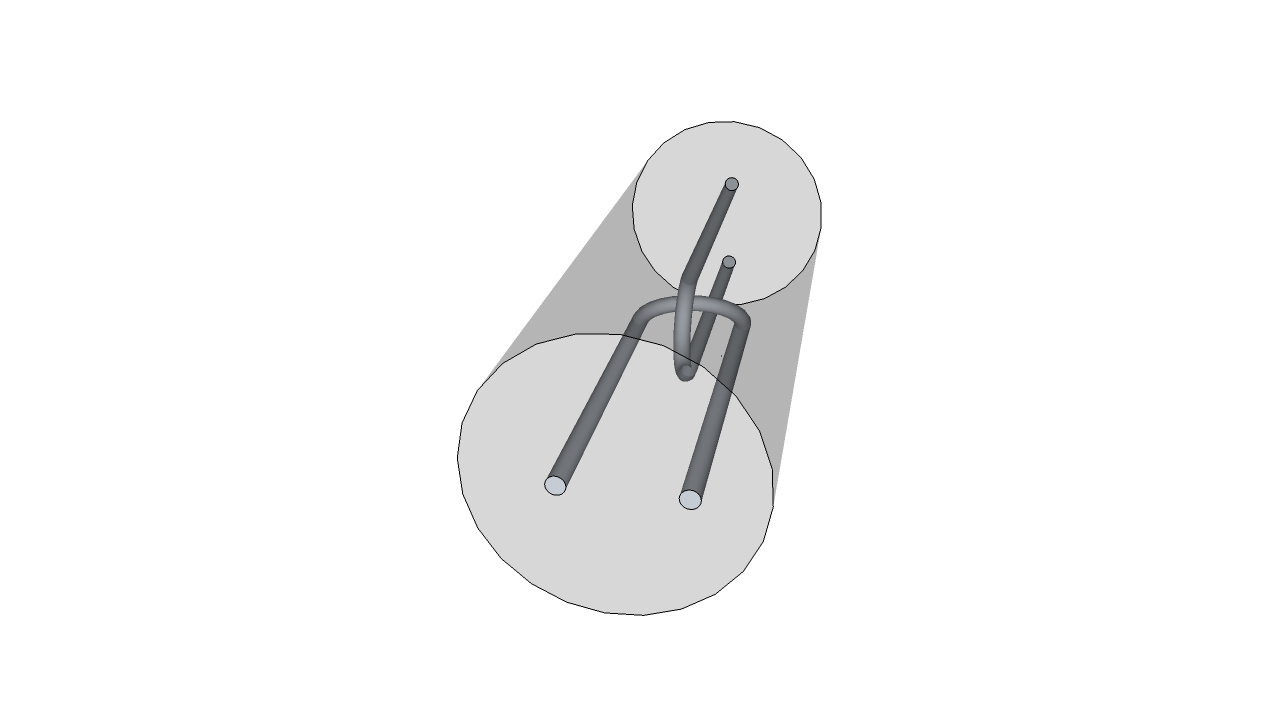} 
  \caption{$[0,1]\times D^2$ with tangle removed}\label{fig7}
 \end{minipage}
\end{figure}

Blow up the singular point of $W_i$ into a 3-ball to obtain a non-singular space $\widetilde W_i$ which maps back to $W_i$ by collapsing the 3-ball into a point. The space $\widetilde W_i$ can be viewed as a subset of $S^3$ whose complement is obtained from the complement of $W_i$ shown in Figure \ref{fig5} by pulling apart the two 2-spheres touching each other in a single point. Turning $\widetilde{W_i}$ inside out with respect to one of the two 2-spheres makes it into
\[
[0,1] \times S^2 = ( [0,1]\times D^2 )\;\cup ([0,1] \times D^2)
\]
with a tubular neighborhood of the clasp tangle removed from one copy of $[0,1] \times D^2$, see Figure \ref{fig7}.

\begin{figure}[!ht]
\centering
\includegraphics[width=0.8\textwidth]{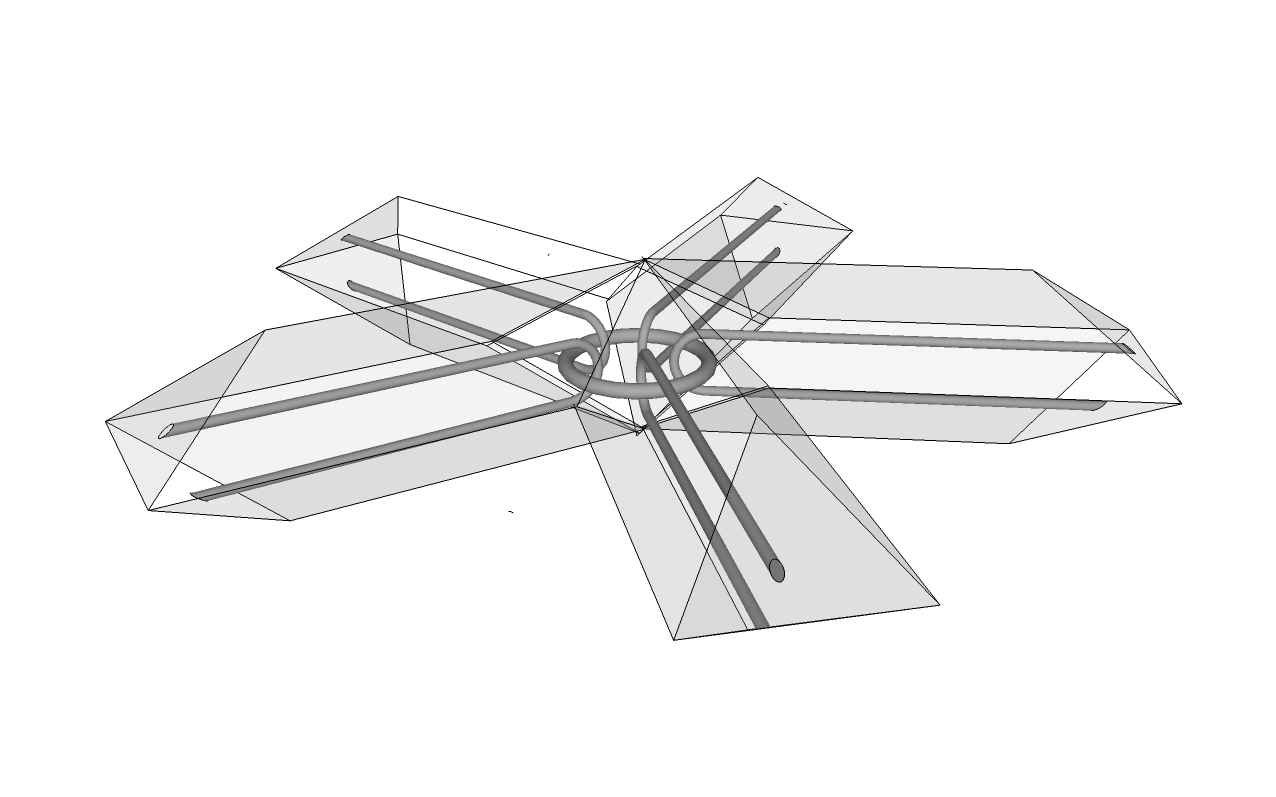}
\caption{Identified Tangle Clasps ($p = 5$)}\label{fig8}
\end{figure}

The spaces $\widetilde{W_i}$ need to be glued back together. We begin by gluing together the $p$ copies of $[0,1]\times D^2$ with the clasp tangles removed, see Figure \ref{fig8}. After identifying the remaining faces intersecting the tangles, we obtain a handlebody of genus $(p-1)$. The tangles form a two-component link which lies inside this handlebody as shown in Figure \ref{fig9}.

\bigskip

\begin{figure}[!ht]
\centering
\includegraphics[width=0.4\textwidth]{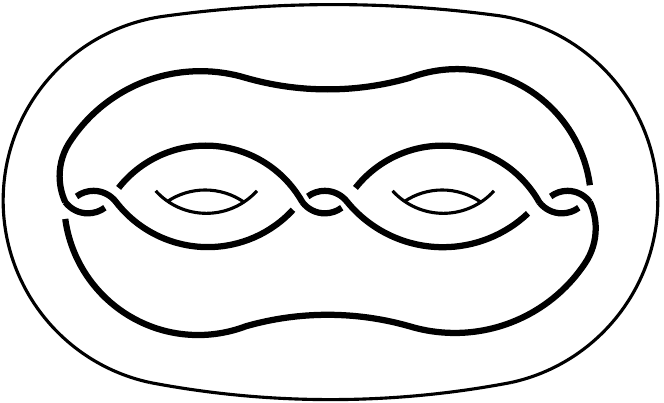}
\caption{The link in a handlebody of genus $(p-1)$ $(p=3)$}
\label{fig9}
\end{figure}

Note that the boundary of this handlebody is mapped into a single point in $Y_0$ hence we can make it into a closed $3$-manifold by attaching another handlebody via an arbitrary homeomorphism $h$ of the boundaries, and map this second handlebody into the same point. We choose $h$ so that the resulting $3$-manifold is the $3$-sphere. Now, filling in the solid tori that were removed from $Y$, we obtain a closed manifold $M$ whose surgery description is shown in Figure \ref{fig10}. 

\medskip

\begin{figure}[!ht]
\centering
\psfrag{0}{$0$}
\includegraphics[width=0.36\textwidth]{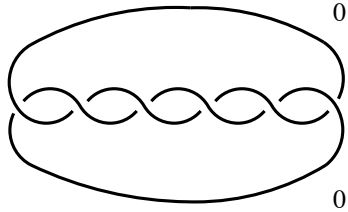}
\caption{Surgery description of $M$ ($p = 3$)}
\label{fig10}
\end{figure}

The $(2,2p)$ torus link in Figure \ref{fig10} may be right-handed (as shown) or left-handed. We will not attempt to pinpoint its handedness because changing it simply reverses the orientation of the manifold $M$, and whether the fundamental class of $M$ realizes the class $[e_1\times e_2 + e_2 \times e_1] \in H_3 (X)$ or its negative will not matter for the applications we have in mind. 


\subsection{Properties of the resolution}
We will next calculate the maps induced by the geometric realization $f = g\circ i: M \to Y\to X$ on the first and third homology groups.

The manifold $M$ is obtained by Dehn surgery on the $(2,2p)$ torus link shown in Figure \ref{fig10}. The meridians of this link, which will be denoted by $\alpha_1$ and $\alpha_2$, generate the cyclic factors in the group $H_1 (M) = \Z/p\,\oplus\,\Z/p$. On the other hand, it follows easily from the description of $Y$ in Section \ref{S:sing-rep} that $H_1(Y) = \Z/p\,\oplus\,\Z/p$, with generators the circle factors of the core torus $S^1 \times S^1 \subset Y$. We will next describe the map $g_*: H_1(M) \to H_1(Y)$ in terms of these generators.

 Up to isotopy, the curve $\alpha_1$ can be thought to wrap once around one meridian of one of the excised tubes in the tangle clasp, see Figure \ref{fig7}, and  $\alpha_2$ wrap around the other tube. This corresponds precisely to wrapping once around each of the excised tori in $A_1\times S^1$. One can further isotope these curves off of the excised tori and directly onto the meridian and the longitude of the core torus in $A_1\times S^1$. This is precisely the core torus in $Y$. The map $g$ will then map $\alpha_1$ and $\alpha_2$ to the standard generators in $H_1(Y) = \Z/p\,\oplus\,\Z/p$.

Keeping in mind that the inclusion $i: Y \to X$ obviously induces an isomorphism in the first homology, we obtain the following result.

\begin{proposition}\label{P:H1}
The map $f_*:\Z/p\,\oplus\,\Z/p\to \Z/p\,\oplus\,\Z/p$ induced by the geometric realization $f = g\circ i: M \to Y \to X$ sends the generators $\alpha_1$ and $\alpha_2$ to the standard generators of $H_1 (X)$.
\end{proposition}

We now consider how $f$ acts on $H_3(M)$. Since $M$ is a closed orientable 3-manifold, we know that $H_3(M)=\Z$.

\begin{proposition}\label{P:H3}
The homomorphism $f_*: H_3 (M) \to H_3 (X)$ induced by the map $f = g\circ i: M \to Y \to X$ is a homomorphism $\Z \to \Z/p$ taking $[M] \in H_3 (M)$ to plus or minus the generator $[e_1\times e_2 + e_2 \times e_1] \in H_3 (X)$.
\end{proposition}

\begin{proof}
Following the calculation in Section \ref{S:sing-rep}, all we need to show is that $g_*: H_3 (M) \to H_3 (Y)$ is an isomorphism $\Z \to \Z$. We will accomplish this by constructing a map $\pi: Y \to S^3$ such that the composition $\pi \circ g: M \to S^3$ has degree $\pm 1$. 

Recall that $Y$ is obtained from its 2-skeleton by attaching 3-cells $e_1\times e_2$ and $e_2\times e_1$. Let $\pi: Y\to S^3$ be the map that collapses all of $Y$ but one of these two 3-cell to a point $p \in S^3$. The pre-image under $\pi$ of any point in $S^3$ other than $p$ contains exactly one point, with the local degree $\pm 1$. Since the map $g: M \to Y$ is a homeomorphism away from the codimension-one singular set of $Y$, the same is true about the map $\pi\circ g$. This means that $\pi\circ g$ is a map of degree $\pm 1$ between closed orientable manifolds $M$ and $S^3$.
\end{proof} 


\section{Cheeger---Simons characters}
In this section, we will study an invariant of CW-complexes derived from the differential characters of Cheeger and Simons \cite{CS}. For a given CW-complex $X$, our invariant will take the shape of a map
\begin{equation}\label{E:cs}
\cs_X: \R(X) \to \Hom (H_3(X),\mathbb R/\Z),
\end{equation}
where $\R(X)$ is the $SU(2)$--character variety of $\pi_1(X)$. The map $\cs$ should be viewed as an extension of the Chern--Simons invariant of 3-dimensional manifolds to manifolds of higher dimensions. We will first give the definition of $\cs_X$ and then proceed to evaluate it explicitly for manifolds $X$ which are squares and deleted squares of lens spaces.


\subsection{Definition of $\cs$}
All manifolds in this section are assumed to be smooth, compact and oriented. Two closed $n$-manifolds $M_1$ and $M_2$ are said to be oriented cobordant if there exists an $(n+1)$--manifold $W$ such that $\p W = -M_1\sqcup M_2$, where $-M_1$ denotes $M_1$ with reversed orientation. The manifold $W$ is then called an oriented cobordism from $M_1$ to $M_2$. The oriented cobordism classes of closed $n$-manifolds form an abelian group with respect to disjoint union. This group is denoted by $\Omega_n$ and called the $n$-th cobordism group.

Given a CW-complex $X$, consider the pairs $(M,f)$, where $M$ is a closed $n$-manifold, not necessarily connected, and $f: M \to X$ a continuous map. Two such pairs, $(M_1,f_1)$ and $(M_2,f_2)$, are said to be cobordant if there is an oriented cobordism $W$ from $M_1$ to $M_2$ and a continuous map $F: W \to X$ which restricts to $f_1$ on $M_1$ and to $f_2$ on $M_2$. The equivalence classes of the pairs $(M,f)$ form an abelian group $\Omega_n (X)$ called the $n$-th oriented cobordism group of $X$. These groups, which are functorial with respect to continuous maps of CW-complexes, give rise to a generalized homology theory called the oriented cobordism theory.

There is a natural homomorphism $\Omega_n (X) \to H_n (X)$ defined by $(M,f)\mapsto f_*\,[M]$, where $[M]$ is the fundamental class of $M$. This homomorphism is known to be an isomorphism for all $n\leq 3$, see for instance Gordon \cite[Lemma 2]{gordon}.  We will use this fact to construct our invariant \eqref{E:cs}.

For any representation $\alpha: \pi_1 X \to SU(2)$, we wish to define a homomorphism $\cs_X(\alpha): H_3 (X) \to \mathbb R/\Z$. Given a homology class $\zeta \in H_3(X)$, use the surjectivity of the map $\Omega_3 (X) \to H_3 (X)$ to find a geometric representative $f: M \to X$ with $f_*\,[M] = \zeta$. Define 
\begin{equation}\label{E:cs-def}
\cs_X (\alpha)(\zeta)\,=\,\cs_M (f^*\alpha),
\end{equation}
where 
\[
\cs_M (f^*\alpha)\;=\;\frac 1 {8\pi^2}\,\int_M\; \tr \left( A\wedge dA + \frac 2 3\;A \wedge A \wedge A\right)
\]

\medskip\noindent
is the value of the Chern-Simons functional on a flat connection $A$ with holonomy $f^*\alpha$. 

\begin{proposition}
The above definition results in a well-defined map $\cs_X: \R(X) \to \Hom\,(H_3(X),\mathbb R/\Z)$.
\end{proposition}

\begin{proof}
Different choices of $A$ and different choices of $\alpha$ within its conjugacy class are known to preserve $\cs_M (f^*\alpha) \pmod \Z$. Since the map $\Omega_3 (X) \to H_3 (X)$ is injective, for any two choices of $f_1: M_1 \to X$ and $f_2: M_2 \to X$ there exists a cobordism $W$ from $M_1$ to $M_2$ and a continuous map $F: W \to X$ restricting to $f_1$ and $f_2$ on the two boundary components. In particular, the pull back representation $F^*\alpha: \pi_1 W \to SU(2)$ restricts to representations $f_1^*\,\alpha: \pi_1 M_1 \to SU(2)$ and $f_2^*\,\alpha: \pi_1 M_2 \to SU(2)$ on the fundamental groups of the two boundary components. This makes $(W,F^*\alpha)$ into a flat cobordism of pairs $(M_1,f_1^*\,\alpha)$ and $(M_2,f_2^*\,\alpha)$. Since the Chern-Simons functional is a flat cobordism invariant \cite{APS:II}, we conclude that $\cs_{M_1} (f_1^*\,\alpha)\,=\,\cs_{M_2} (f_2^*\,\alpha)$, making $\cs_X (\alpha)$ well-defined. That $\cs_X (\alpha)$ is a homomorphism follows easily from the fact that a geometric representative for a sum of homology classes can be chosen to be a disjoint union of geometric representatives of the summands. 
\end{proof}

\begin{proposition}\label{P:naturality}
For any continuous map $h: Y \to X$ the following diagram commutes

\[
\begin{CD}
\R(X) @>\cs_X>> \Hom\, (H_3(X),\mathbb{R}/\Z)\\
@VV h^* V @VV h^* V\\
\R(Y) @>\cs_Y>>  \Hom\, (H_3(Y),\mathbb{R}/\Z)
\end{CD}
\]
\smallskip
\end{proposition}

\begin{proof}
Given a representation $\alpha: \pi_1 X \to SU(2)$ and a homology class $\zeta \in H_3 (Y)$, we wish to show that $\cs_Y (h^*\alpha) (\zeta)\,=\,\cs_X(\alpha)(h_*\zeta)$. Choose a geometric representative $f: M \to Y$ for the class $\zeta$ so that $f_*\,[M] =\zeta$ then the composition $h\circ f: M \to Y \to X$ will give a geometric representative for the class $h_*\zeta$ because $(h\circ f)_*\,[M] = h_* (f_*\,[M]) = h_*\zeta$. The formula now follows because one can easily see that both sides of it are equal to $\cs_M (f^*h^*\alpha)$.
\end{proof}

\begin{remark}
An equivalent way to define the invariant \eqref{E:cs-def} is by pulling back the universal Chern--Simons class $\cs: H_3 (BSU(2)^{\delta}) \to \mathbb R/\Z$, see for instance \cite[Section 5]{KK}. To be precise, let $SU(2)^{\delta}$ stand for $SU(2)$ with the discrete topology and $BSU(2)^{\delta}$ for its classifying space. Every representation $\alpha: \pi_1 X \to SU(2)$ gives rise to a continuous map $X \to BSU(2)^{\delta}$ which is unique up to homotopy equivalence. Then $\cs_X (\alpha): H_3 (X) \to \mathbb R/\Z$ is the pull back of the universal Chern--Simons class $\cs$ via the induced homomorphism $H_3 (X) \to H_3 (BSU(2)^{\delta})$.
\end{remark}


\subsection{Squares of lens spaces}
Given a lens space $L = L(p,q)$, consider its square $X = L\,\times\,L$ and let $1 \in \pi_1 (L) = \Z/p$ be the canonical generator. Since $\pi_1 (X) = \Z/p\,\oplus\,\Z/p$ is abelian, every representation $\alpha: \pi_1 (X) \to SU(2)$ can be conjugated to a representation $\alpha (k,\ell): \Z/p\,\oplus\,\Z/p \to U(1)$ sending the canonical generators of the fundamental groups of the two factors to $\exp(2\pi ik/p)$ and $\exp(2\pi i\ell/p)$ with $0 \le k, \ell \le p-1$. Among these abelian representations, the only ones conjugate to each other are $\alpha(k,\ell)$ and $\alpha(p-k,p-\ell)$, via the unit quaternion $j \in SU(2)$. Therefore, the character variety $\R(X)$ consists of $(p^2+1)/2$ points for $p$ odd and $(p^2+4)/2$ points for $p$ even. For each of these representations, we wish to compute the map:
\[
\cs(\alpha(k,\ell)): H_3 (X) \to \mathbb R/\Z.
\]
Since $\cs(\alpha(k,\ell))$ is a homomorphism, it will be sufficient to compute it on a set of generators of $H_3(X) = \Z\,\oplus\,\Z\,\oplus\,\Z/p$, see Section \ref{S:squares}.

We begin with the infinite cyclic summands in $H_3(X)$, realized geometrically by embedding $L$ into $X$ as the factors $L\,\times\,e_0$ and $e_0\,\times\,L$. 

\begin{proposition}\label{P:cs-one}
The values of $\cs(\alpha(k,\ell))$ on the above free generators of $H_3 (X)$ are $-k^2r/p$ and $-\ell^2r/p$, where $r$ is any integer such that $qr = -1\pmod p$.
\end{proposition}

\begin{proof}
With respect to the aforementioned embedding, the representation $\alpha(k,\ell)$ pulls back to representations $\pi_1 (L) \to SU(2)$ sending the canonical generator to $\exp(2\pi ik/p)$ and $\exp(2\pi i\ell/p)$, respectively. The Chern--Simons invariants of such representations were computed explicitly in Kirk--Klassen \cite[Theorem 5.1]{KK}. 
\end{proof}

The torsion part of $H_3 (X)$ is generated by the class $[e_2\, \times\, e_1 + e_1\, \times\, e_2]$ which was realized, up to a sign, by a continuous map $f: M \to X$ in Section \ref{S:sing}. A surgery description of the manifold $M$ is shown in Figure \ref{fig10}. We wish to compute the Chern--Simons function on $M$ for the induced representations $f^*\alpha(k,\ell): \pi_1 (M) \to SU(2)$. 

\begin{lemma}\label{L:seifert}
The manifold $M$ is Seifert fibered over the 2-sphere, with the unnormalized Seifert invariants $(p,-1)$, $(p,-1)$, and $(p,1)$.
\end{lemma}

\begin{proof}
The Seifert fibered manifold with the unnormalized Seifert invariants $(p,-1)$, $(p,-1)$, and $(p,1)$ can be obtained by plumbing on the diagram shown in Figure \ref{fig11}. 
\smallskip

\begin{figure}[!ht]
\centering
\includegraphics[width=0.45\textwidth]{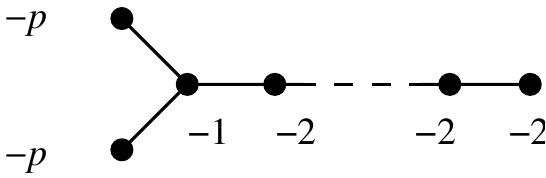}
\caption{Plumbing diagram}
\label{fig11}
\end{figure}

\noindent
One can easily see that blowing down the chain of $p$ vertices framed by $-1$ and $-2$ results in the framed link shown in Figure \ref{fig10}.
\end{proof}

\begin{proposition}\label{P:cs-two}
The function $\cs(\alpha(k,\ell))$ takes value $\pm\, 2k\ell/p$ on the generator $[e_2 \times e_1 + e_1 \times e_2] \in H_3 (X)$.
\end{proposition}

\begin{proof}
This can be done using Auckly's paper \cite{auckly} and the fact that the manifold $M$ is Seifert fibered, see Lemma \ref{L:seifert}. To be precise, consider the standard presentation
\[
\langle\;h,x_1,x_2,x_3\;|\; h\;\text{central},\, x_1^p=h,\, x_2^p=h,\, x_3^p = h^{-1},\, x_1 x_2 x_3 = 1\;\rangle
\] 
of the fundamental group of $M$, where $x_1$, $x_2$ and $x_3$ are meridians of the circles in Figure ? framed respectively by $-p$, $-p$ and $p$. One can easily check that the representation $\alpha(k,\ell)$ acts on the generators as follows\,:
\[
h \to 1,\quad x_1\to e^{2\pi i k/p},\quad x_2\to e^{2\pi i \ell/p},\quad x_3\to e^{-2\pi i (k+\ell)/p}.
\]
Therefore, $\alpha(k,\ell) = \omega(0,k,\ell,k+\ell)[1,1,j]$ in Auckly's terminology on page 231 of \cite{auckly}, and the Chern--Simons function for this representation computed on page 232 of \cite{auckly} equals
\[
-\frac {k^2} p - \frac {\ell^2} p + \frac {(k+\ell)^2} p = \frac {2k\ell} p.
\]
Since $M$ realizes $[e_2 \times e_1 + e_1 \times e_2]$ up to sign, the result follows.
\end{proof}


\subsection{Deleted squares of lens spaces}
We will next compute the map $\cs_{X_0}: \R(X_0) \to \Hom(H_3(X_0),\mathbb R/\Z)$ for deleted squares $X_0$ of lens spaces. Since the inclusion $i: X_0 \to X$ induces an isomorphism of fundamental groups, every $SU(2)$ representation of $\pi_1(X_0)$ is the pull back via $i$ of a representation $\alpha(k,\ell): \pi_1 (X) \to SU(2)$. Choose the generators in $H_3 (X_0) = \Z\,\oplus\,\Z/p$ as in Lemma \ref{L:gens} so that the map $i_*: H_3 (X_0) \to H_3 (X)$ has the property that $i_*(1,0) = (1,1,0)$ and $i_*(0,1) = (0,0,1)$ with respect to the canonical basis of $H_3 (X) = \Z\,\oplus\,\Z\,\oplus\,\Z/p$.

\begin{proposition}\label{P:cs-del}
Let $r$ be any integer such that $qr = -1\pmod p$. Then the values of $\cs_{X_0}(\alpha(k,\ell)): H_3 (X_0) \to \mathbb R/\Z$ on the above generators of $H_3 (X_0)$ are given by the formulas 
\[
\cs_{X_0}(i^*\alpha(k,\ell))(1,0) = -r(k^2+\ell^2)/p\quad\text{and}\quad \cs_{X_0}(i^*\alpha(k,\ell))(0,1) = \pm\, 2k\ell/p.
\]
\end{proposition}

\begin{proof}
This follows from the naturality property of the Chern--Simons map, see Proposition \ref{P:naturality}, and the calculation of $\cs_X$ in Propositions \ref{P:cs-one},
\begin{multline}\notag
\cs_{X_0}(i^*\alpha(k,\ell))(1,0) = \cs_X(\alpha(k,\ell))(1,1,0) \\ = \cs_X(\alpha(k,\ell))(1,0,0) + \cs_X(\alpha(k,\ell))(0,1,0) = -r(k^2 + \ell^2)/p
\end{multline}
and in Proposition \ref{P:cs-two},
\[
\cs_{X_0}(i^*\alpha(k,\ell))(0,1) = \cs_X(\alpha(k,\ell))(0,0,1) = \pm\,2k\ell/p.
\]
\end{proof}

From now on, we will choose generators in $\pi_1 (X_0) = \Z/p\,\oplus\,\Z/p$ so that $i_*: \pi_1 (X_0) \to \pi_1 (X)$ is the identity map, and denote $i^*\alpha(k,\ell): \pi_1 (X_0) \to SU(2)$ by simply $\alpha(k,\ell)$.


\section{Homotopy equivalence of deleted squares}\label{S:heq}
Let $X_0$ and $X'_0$ be the deleted squares of lens spaces $L(p,q)$ and $L(p',q')$, respectively. Suppose that $X_0$ and $X'_0$ are homotopy equivalent via a homotopy equivalence $f: X'_0 \to X_0$. Since $f_*: \pi_1 X'_0 \to \pi_1 X_0$ is an isomorphism, we immediately conclude that $p = p'$. Using the canonical isomorphisms $\pi_1 X_0 = \Z/p\,\oplus\,\Z/p$ and $\pi_1 X'_0 = \Z/p\,\oplus\,\Z/p$, the isomorphism $f_*: \pi_1 X'_0 \to \pi_1 X_0$ is given by a matrix 
\begin{equation}\label{E:pi1}
f_* = \begin{pmatrix} \alpha & \gamma \\ \beta & \delta \end{pmatrix}\; \in\; GL (2,\Z/p).
\end{equation}
In particular, for any representation $\alpha(k,\ell): \pi_1 X_0 \to SU(2)$, we have $f^*\alpha(k,\ell) = \alpha(k',\ell')$, where 
\begin{equation}\label{E:kl}
k' = \alpha k + \beta \ell \quad\text{and} \quad \ell' = \gamma k + \delta \ell.
\end{equation}

Next, the naturality property of Proposition \ref{P:naturality} implies that we have a commutative diagram 

\[
\begin{CD}
\R(X_0) @> \quad \cs_{X_0} \quad >> \Hom(H_3(X_0),\mathbb{R}/\Z)\\
@VV f^* V @VV f^* V\\
\R(X'_0) @> \quad \cs_{X'_0} \quad >>  \Hom(H_3(X'_0),\mathbb{R}/\Z)
\end{CD}
\]

\bigskip\noindent
With respect to the generators of $H_3 (X_0) = \Z\,\oplus\,\Z/p$ and $H_3 (X'_0) = \Z\,\oplus\,\Z/p$ chosen as in Lemma \ref{L:gens}, the isomorphism $f_*: H_3 (X'_0) \to H_3 (X_0)$ can be described by a matrix
\begin{equation}\label{E:f*}
f_*\;=\;\begin{pmatrix} \ep & 0 \\ a & b\end{pmatrix},
\end{equation}
where $\ep = \pm 1$ and $b$ is a unit in the ring $\Z/p$. The commutativity of the above diagram means exactly that, for any choice of $k$ and $\ell$, one has the commutative diagram

\[
\begin{CD}
H_3 (X'_0) @> \qquad \cs_{X'_0} (\alpha(k',\ell')) \qquad >> \mathbb R/\Z \\
@V f_* VV @VV= V\\
H_3 (X_0) @> \qquad \cs_{X_0} (\alpha(k,\ell)) \qquad >> \mathbb R/\Z 
\end{CD}
\]

\bigskip\noindent
where the integers $k',\ell'$ are given by the equations \eqref{E:kl} and the isomorphism $f_*$ by the equation \eqref{E:f*}. 

By evaluating on the generators $(1,0) \in H_3 (X'_0)$ and $(0,1) \in H_3 (X'_0)$ and using Proposition \ref{P:cs-del}, we easily conclude that the commutativity of the above diagram is equivalent to the following equations holding true for any choice of integers $k$ and $\ell$\,:
\medskip
\begin{gather}
-\ep r (k^2 + \ell^2) \pm 2ak\ell = -r'((\alpha k+ \beta \ell)^2 + (\gamma k + \delta \ell)^2)\pmod p \label{E:one} \\
\pm\,2bk\ell = 2 (\alpha k + \beta \ell)(\gamma k + \delta \ell) \pmod p \label{E:two}
\end{gather}

\medskip

\begin{proposition}
If the deleted squares of lens spaces are homotopy equivalent then the lens spaces themselves are homotopy equivalent. The latter homotopy equivalence is orientation preserving if and only if $\ep = 1$ in \eqref{E:f*}.
\end{proposition}

\begin{proof}
Set $k = 1$ and $\ell = 0$ in equations \eqref{E:one} and \eqref{E:two} to obtain $-\ep r = -r' (\alpha^2 + \gamma^2)\pmod p$ and $2\alpha\gamma = 0 \pmod p$ and consequently $-\ep r = -r' (\alpha + \gamma)^2\pmod p$. Multiplying the latter equation by $qq'$ and keeping in mind that $qr = -1\pmod p$ and $q'r'=-1\pmod p$ we conclude that
\[
\ep q' = q (\alpha + \gamma)^2\pmod p.
\]
Therefore, $L(p,q)$ and $L(p,q')$ are homotopy equivalent via a homotopy equivalence which is orientation preserving if and only if $\ep = 1$.
\end{proof}

\begin{proposition}
Assume that $p$ is an odd prime, and that the deleted squares $X_0$ and $X'_0$ of lens spaces $L(p,q)$ and $L(p,q')$ are homotopy equivalent via a homotopy equivalence $f: X'_0 \to X_0$. Then 
the matrix \eqref{E:pi1} of the induced homomorphism $f_*: \pi_1 (X'_0) \to \pi_1 (X_0)$ with respect to the canonical bases of the fundamental groups is of the form 
\medskip
\[
\begin{pmatrix} \alpha & 0 \\ 0 & \pm\,\alpha \end{pmatrix}\;\;\text{or}\;\;
\begin{pmatrix} 0 & \alpha \\ \pm\,\alpha & 0 \end{pmatrix}\quad\text{with}\quad \ep q'\, =\, q\alpha^2\pmod p,
\]

\medskip\noindent
and the matrix \eqref{E:f*} of the induced homomorphism $f_*: H_3 (X'_0) \to H_3 (X_0)$ is of the form 
\[
\begin{pmatrix} \ep & 0 \\ 0 & \pm\,\alpha^2 \end{pmatrix}.
\]
\end{proposition}

\begin{proof}
Setting first $k = 1$, $\ell = 0$ and then $k = 0$, $\ell = 1$ in equation \eqref{E:two} results in equations $\alpha\gamma = 0\pmod p$ and $\beta\delta = 0\pmod p$. Since the determinant $\alpha\delta - \beta\gamma$ of the matrix \eqref{E:pi1} is a unit in $\Z/p$, we readily conclude that there are only two options for the matrix \eqref{E:pi1}, one with $\alpha =\delta = 0$ and the other with $\beta = \gamma = 0$. 

Let us first consider the case of $\beta =\gamma = 0$. Plugging these values in equation \eqref{E:one} and evaluating at $k = 1$, $\ell = 0$ and then at $k = 0$, $\ell = 1$ results in two equations, $\ep r = r'\alpha^2 \pmod p$ and $\ep r = r'\delta^2 \pmod p$. These imply that $\alpha^2 = \delta^2\pmod p$ and therefore $\delta = \pm \alpha$. That $b = \pm \alpha^2$ now follows by plugging $\beta =\gamma = 0$ in equations \eqref{E:two}. The case of $\alpha =\delta = 0$ is treated similarly.

That $a = 0$ follows by setting first $k = 1$, $\ell = 1$, and then $k = 1$, $\ell = -1$, in equation \eqref{E:one}. 
\end{proof}

\begin{remark}\label{R:diag}
By composing the homotopy equivalence $f: X'_0 \to X_0$ with the homeomorphism $g: X_0 \to X_0$ given by $g(x,y) = (y,x)$ if necessary, one may assume that the induced homomorphism $f_*: \pi_1 (X'_0) \to \pi_1 (X_0)$ is given by the matrix
\begin{equation}\label{E:diag}
\begin{pmatrix} \alpha & 0 \\ 0 & \pm\,\alpha \end{pmatrix}.
\end{equation}

\medskip\noindent
That the sign in this matrix is actually plus, as we claimed in the theorem of the introduction, will follow from an argument in the next section.
\end{remark}

\begin{remark}
The techniques of this section could also be applied to homotopy equivalences between the squares of lens spaces, with the diagonal left intact. One can easily see that these techniques produce no new algebraic restrictions, which should not be surprising since the squares of any two lens spaces $L(p,q)$ and $L(p,q')$ are known to be not just homotopy equivalent but diffeomorphic to each other; see Kwasik--Schultz \cite{KS}.
\end{remark}


\section{Universal covers of deleted squares}\label{S:covers}
We wish next to combine our results with those of Longoni and Salvatore \cite{SL} to obtain further obstructions to the existence of homotopy equivalences $f: X'_0 \to X_0$ of deleted squares. The approach taken in \cite{SL} was to lift $f$ to a homotopy equivalence $\tilde f: \tilde X'_0 \to \tilde X_0$ of the universal covering spaces and study its effect on the triple Massey products in cohomology.  We begin in this section by computing the cohomology ring $H^*(\tilde X_0)$ of a deleted square together with its natural $\Z\,[\pi_1 (X_0)]$ module structure.

To describe the universal covering space $\tilde X_0$, we will follow \cite{SL} by first observing that the universal covering space of the square $X = L \times L$ of a lens space $L = L(p,q)$ is the product $S^3 \times S^3$ on which $(m,n) \in \pi_1 (X) = \Z/p\,\oplus\,\Z/p$ acts by the formula 
\[
\tau_{m,n}\, ((z_1,z_2),(z_3,z_4))\, =\, ((\zeta^m z_1,\zeta^{qm} z_2),(\zeta^{n} z_3,\zeta^{qn} z_4)),\quad \zeta = e^{2\pi i/p}.
\]
The orbit of the diagonal $\Delta \subset S^3 \times S^3$ under this action consists of the disjoint 3-spheres 
\[
\Delta_k = \{((z_1,z_2),(z_3,z_4))\; |\; (z_1,z_2) = (\zeta^k z_3,\zeta^{qk} z_4) \},\quad k \in \Z/p, 
\]
embedded into $S^3 \times S^3$. Their removal from $S^3 \times S^3$ results in the so called orbit configuration space 
\[
\tilde{X_0 }=\{((z_1,z_2),(z_3,z_4))\;|\;(z_1,z_2)\neq(\zeta^k z_3,\zeta^{qk} z_4)\;\;\text{for any}\; k\in \Z/p \},
\]
which is the universal covering space of $X_0$. In short, 
\[
\tilde X_0\; =\; (S^3 \times S^3)\; \setminus\; \left(\;\bigsqcup\; \Delta_k\right).
\]

\begin{lemma}\label{L:cup}
The only non-trivial reduced cohomology groups of $\tilde X_0$ are $H^2 (\tilde X_0) = \Z^{\,p-1}$, $H^3 (\tilde  X_0) = \Z$ and $H^5 (\tilde X_0) = \Z^{\,p-1}$. The cup-product with a generator of $H^3 (\tilde X_0)$ provides an isomorphism $H^2 (\tilde X_0) \to H^5 (\tilde X_0)$.
 \end{lemma}

\begin{proof}
This follows from the Leray--Serre spectral sequence of the bundle $\tilde X_0 \to S^3$ obtained by projecting $\tilde X_0 \subset S^3 \times S^3$ onto the first factor. The fiber of this bundle is the $3$-sphere with $p$ punctures, which is homotopy equivalent to a one-point union of $p-1$ copies of $S^2$. The bundle admits a section forcing the spectral sequence to collapse. Since there is no extension problem involved, the cohomology ring $H^*(\tilde X_0)$ splits as a tensor product. 
\end{proof}

Longoni and Salvatore \cite{SL} provided an explicit set of generators in $H^2 (\tilde X_0)$ using the Poincar\'e duality isomorphism 
\medskip
\[
H^2 \left(S^3 \times S^3 \setminus \left(\;\bigsqcup\; \Delta_k\right)\right) = H_4 \left(S^3\times S^3,\, \bigsqcup\;\Delta_k\right).
\]

\medskip\noindent
The generators $a_0,\ldots,a_{p-1}$ are defined as the Poincar\'e duals of the sub-manifolds 
\[
A_k = \{((z_1,z_2),(\zeta^s z_1,\zeta^{qs} z_2))\in S^3\times S^3\;|\; s \in [k-1, k]\,\}\;\subset\;S^3 \times S^3,
\]
and the only relation they satisfy is $a_0 + a_1 + \ldots + a_{p-1} = 0$. 

\begin{lemma}\label{L:diag}
The fundamental group $\pi_1 (X_0) = \Z/p\,\oplus\,\Z/p$ acts on the generators $a_k$ by the rule 
$\tau^*_{m,n}\,(a_k) = a_{k + m - n}$, where the indices are computed mod $p$.
\end{lemma}

\begin{proof}
The action of $\pi_1 (X_0)$ on the sub-manifolds $A_k$ was computed in Miller \cite[Section 2.1]{miller} to be $\tau_{m,n}\,(A_k) = A_{k+n-m}$. Passing to the Poincar\'e duals $a_k$ in the commutative diagram 
\medskip
\[
\begin{CD}
H^2(\tilde X_0) @> \tau^*_{m,n} >> H^2(\tilde X_0)\\
@V \PD VV @V \PD VV\\
H_4\left(S^3\times S^3,\;\bigsqcup\,\Delta_k\right)  @< \tau_{m,n} << H_4\left(S^3\times S^3,\;\bigsqcup\,\Delta_k\right),
\end{CD}
\]

\bigskip\noindent
where $\PD$ stands for the Poincar\'e duality isomorphism, and using the fact that the inverse of $\tau_{m,n}$ is $\tau_{-m,-n}$ gives the desired formula.
\end{proof}

We now wish to describe the behavior with respect to the above action of the homomorphism $\tilde f^*$ induced on the second cohomology of the universal covers by a homotopy equivalence $f: X'_0 \to X_0$.

\begin{proposition}\label{P:plus}
Let $f: X_0' \to X_0$ be a homotopy equivalence such that the induced homomorphism $f_*: \pi_1 (X'_0) \to \pi_1 (X_0)$ is given by the matrix \eqref{E:diag} with $\ep q' = q\alpha^2\pmod p$. Then the sign in the matrix \eqref{E:diag} is a plus and the following diagram commutes
\[
\begin{CD}
H^2 (\tilde X_0) @> \tilde f^* >> H^2 (\tilde X'_0) \\
@V \tau^*_{\alpha m,\alpha n} VV @VV \tau^*_{m,n} V \\
H^2 (\tilde X_0)  @> \tilde f^* >> H^2 (\tilde X'_0).
\end{CD}
\]
\end{proposition}

\smallskip

\begin {proof} 
The fundamental group acts by deck transformations on the universal cover giving rise to the commutative diagram
\medskip
\[
\begin{CD}
\pi_1(X'_0) @>\tau' >> \Aut\,(H^2(\tilde X'_0))\\
@Vf_*VV @VV \Ad\,(\tilde f^*)V \\
\pi_1(X_0)  @>\tau >> \Aut\,(H^2(\tilde X_0))
\end{CD}
\]

\bigskip\noindent
where $\tau(m,n) = \tau_{m,n}$ and similarly for $\tau'$. The commutativity of this diagram implies, in particular, that $f_*(\ker \tau') = \ker \tau$. However, according to Lemma \ref{L:diag}, the kernels of both $\tau$ and $\tau'$ are the diagonal subgroups of $\Z/p\,\oplus\,\Z/p$, therefore, the matrix \eqref{E:diag} must be diagonal. 
\end {proof}

According to Lemma \ref{L:diag}, the diagonal subgroup $\Z/p\subset \Z/p\,\oplus\,\Z/p$ of $\pi_1(X_0)$ acts trivially on $H^2 (\tilde X_0)$ thereby giving rise to an effective action of the quotient group $\Z/p$ on $H^2 (\tilde X_0)$. Let $t = (1,0) \in \Z/p\,\oplus\,\Z/p$ be a generator of this group, and consider the cyclotomic ring $\Z[t]/(N)$, where $N = 1+t+\ldots+t^{p-1}$.

\begin{corollary}
The homomorphism $\phi: H^2 (\tilde X_0) \to \Z[t]/(N)$ of abelian groups defined on the generators by the formula $\phi(a_{k+1}) = t^k$ is an isomorphism of $\Z[t]$--modules. 
\end{corollary}

Using this module structure, the result of Proposition \ref{P:plus} can be re-stated as follows.

\begin{corollary}\label{C:heq}
Any homotopy equivalence $f: X'_0 \to X_0$ induces an isomorphism $\tilde f^*: H^2 (\tilde X_0) \to H^2 (\tilde X'_0)$ which makes the following diagram commute
\medskip
\[
\begin{CD}
H^2 (\tilde X_0) @> \tilde f^* >> H^2 (\tilde X'_0) \\
@V t^{\alpha} VV @VV t V \\
H^2 (\tilde X_0) @> \tilde f^* >> H^2 (\tilde X'_0).
\end{CD}
\]
\end{corollary}

\bigskip


\section{Massey products}\label{S:massey}
In this section, we will combine our results with the Massey product  techniques of \cite{SL} to obtain further restrictions on possible homotopy equivalences between deleted squares of lens spaces. 


\subsection{Definition}
Let $L = L(p,q)$ be a lens space and $\tilde X_0$ the universal covering space of its deleted square $X_0$. It follows from Lemma \ref{L:cup} that for any cohomology classes $x$, $y$, $z \in H^2 (\tilde X_0)$ their pairwise cup-products vanish, giving rise to the Massey products $\langle x, y, z\rangle$. The precise definition is as follows. Choose singular cochains $\bar x$, $\bar y$ and $\bar z$ representing the cohomology classes $x$, $y$ and $z$, respectively.  Since $x\,\cup\,y = y\,\cup\,z = 0$ there exist singular cochains $Z$ and $X$ such that $dZ = \bar x\,\cup\,\bar y$ and $dX = \bar y\,\cup\,\bar z$, and we define $\langle x,y,z \rangle$ to be the cohomology class of the cocycle $Z\,\cup\,\bar z - \bar x\,\cup\,X$. Since the choice of $Z$ and $X$ is not unique, $\langle x, y, z\rangle$ is only well defined as a class in $H^*(\tilde X_0)/\langle x, z\rangle$, where $\langle x, z\rangle$ is the ideal generated by $x$ and $z$. The class $\langle x, y, z\rangle$ obviously has degree $5$. 


\subsection{Calculations}
From now on, we will work with $\Z/2$ coefficients. Assuming that $p$ is an odd prime, we will use results of Section \ref{S:covers} to identify both $H^2 (\tilde X_0;\Z/2)$ and $H^5 (\tilde X_0;\Z/2)$ with the additive group of the cyclotomic field
\[
R\; =\; (\Z/2)\,[t]\,/\,(1 + t + \ldots + t^{p-1}), 
\]
the latter via the cup-product with the canonical generator of $H^3 (\tilde X_0;\Z/2)$. The Massey product can then be viewed as a multi-valued ternary operation 
\begin{equation}\label{E:RRR}
\mu:\, R \times R \times R \longrightarrow  R
\end{equation}
whose value $\mu(x,y,z) = \langle x, y, z \rangle$ on a triple $(x, y, z)$ is only well defined up to adding an arbitrary linear combination of $x$ and $z$. 

The Massey products \eqref{E:RRR} were calculated by Miller \cite[Theorem 3.33]{miller} for all $L(p,q)$ such that $0 < q < p/2$ using the intersection theory of \cite{SL}. The theorem below summarizes Miller's calculation. To state it, we need some definitions.

For any interval $(a,b) \subset \mathbb R$ denote by $(a,b)_{S^1}$ its projection to the circle $S^1$ 
viewed as the quotient of $\mathbb R$ by the subgroup $p\,\mathbb Z$. Given $j$ and $k$ in $\Z/p$, we say that $k$ is an interloper of $j$, denoted by $k \prec j$, if  
\smallskip
\[
\begin{cases}
\quad jq \in (0,q)_{S^1}\;\;\text{and}\;\; k\in[j,p],\;\;\text{or} \\
\; -jq \in (0,q)_{S^1}\;\;\text{and}\;\;k\in[0,j].
\end{cases}
\]

\smallskip\noindent
By definition, integers $j$ such that neither $jq$ nor $-jq$ belong to $(0,q)_{S^1}$ have no interlopers.

\smallskip

\begin{theorem}\label{T:miller}
The Massey product structure \eqref{E:RRR} is completely determined by the linearity and the Massey products $\langle t^k,t^{\ell},t^j\rangle$ on the monomials which, before factoring out the ideals, obey the following relations\,:
\medskip
\begin{itemize}
\item $\langle t^{k+n},t^{\ell+n},t^{j+n}\rangle = t^n\cdot \langle t^k,t^{\ell},t^j\rangle$
\medskip
\item $\langle t^{k},t^{\ell},t^{j}\rangle = \langle t^j,t^{\ell},t^k\rangle$
\medskip
\item Assuming $j \neq 0$,

\bigskip

$\langle 1,1,t^{j} \rangle\, =\,
\begin{cases} \; t\;\, + \ldots +\; t^j, &\; \text{if}\; -jq \in (0,q)_{S^1}\\
		      \; t^j + \ldots + t^{p-1},   &\; \text{if}\quad\; jq \in (0,q)_{S^1}\\
                      \; 0, &\; \text{otherwise}
\end{cases}$
\bigskip
\item Assuming $j \neq 0$ and $k \neq 0$,

\bigskip

$\langle t^k,1,t^{j}\rangle =
\begin{cases} \; t^k + t^j, &\; \text{if}\quad k \prec j\;\; \text{and}\;\; j \prec k\\
	              \quad t^k,  &\; \text{if}\quad k \prec j\;\; \text{and}\;\; j \nprec k \\
                       \quad t^j,      &\; \text{if}\quad j \prec k\;\; \text{and}\;\; k \nprec j \\
   		      \quad  0,      &\; \text{otherwise}
\end{cases}$
\end{itemize}
\end{theorem}

\smallskip

\begin{example}
Let $L = L(5,2)$. The condition $\pm\,2 j \in (0,2)_{S^2}$ is only satisfied when $j = 2$ and $j = 3$. In the former case, $-2j = -4 = 1 \in (0,2)_{S^1}$, and in the latter $2j = 6 = 1 \in (0,2)_{S^1}$. The interlopers of $j = 2$ are $k = 0, 1, 2$, and those of $j = 3$ are $k = 3, 4, 0$. According to the above theorem, $\langle 1,1,t^2 \rangle = t + t^2$, $\langle 1,1,t^3 \rangle = t^3 + t^4$, $\langle t,1,t^2 \rangle = t$, $\langle t^4,1,t^3 \rangle = t^4$, up to the indeterminacy.
\end{example}


\subsection{Naturality}\label{S:natural}
Let $f: X_0' \to X_0$ be a homotopy equivalence between deleted squares of lens spaces as in Section \ref{S:heq}. Using identifications of $H^2 (\tilde X_0;\Z/2)$ and $H^2 (\tilde X'_0;\Z/2)$ with the cyclotomic ring $R$, we will view the corresponding Massey products as ternary operations $\mu: R \times R \times R \to R$ and $\mu': R \times R \times R \to R$. Since Massey products are natural with respect to homotopy equivalences, the following diagram 
\smallskip
\[
\begin{CD}
R \times R \times R @ > \mu >> R \\
@V \tilde f^* \times \tilde f^* \times \tilde f^* VV @VV \tilde f^* V \\
R \times R \times R @ > \mu' >> R \\
\end{CD}
\]

\medskip\noindent
must commute up to indeterminacy in the Massey products. Note that the map $\tilde f^*$ in this diagram is an isomorphism of abelian groups which, according to Corollary \ref{C:heq}, makes the following diagram commute 
\smallskip
\[
\begin{CD}
R @> \tilde f^* >> R \\
@V t VV @VV t^{\beta} V \\
R @> \tilde f^* >> R
\end{CD}
\]

\medskip\noindent
where $\alpha\cdot \beta = 1\mod p$. Taking all of the above into account, we conclude that the homomorphism $\tilde f^*$ is uniquely determined by the polynomial $\tilde f^*(1) \in R$ and, for any choice of $k, \ell \pmod p$, satisfies the relation
\smallskip
\begin{equation}\label{E:massey}
\begin{array}{lr}
\mu' \left(t^{\beta k} \tilde f^*(1),\, \tilde f^*(1),\, t^{\beta \ell} \tilde f^*(1) \right) = \tilde f^* \left(\mu\,(t^k, 1, t^{\ell})\right) \\ \hspace{1.2in} +\, (a\cdot t^{\beta k} + b\cdot t^{\beta \ell})\cdot \tilde f^* (1)\;\;\text{for some\; $a, b \in \Z/2$}.
\end{array}
\end{equation}

\smallskip\noindent
Writing $\tilde f^*(1) \in R$ as a polynomial with undetermined coefficients and using the Massey product formulas of Theorem \ref{T:miller}, one can attempt solving this system of equations using a computer. The {\tt Maple} worksheet we used can be obtained from either of the authors.

\begin{example}
As a warm up exercise, let $qq' = 1$ (mod $p$) and consider the homeomorphism $L(p,q') \to L(p,q)$ sending $(z_1,z_2)$ to $(z_2,z_1)$. It gives rise to a homeomorphism $f: X'_0 \to X_0$ of deleted squares and a homeomorphism $\tilde f$ of their universal covers. The map $\tilde f^*$ can be found using the commutative diagram
\[
\begin{CD}
H^2(\tilde X_0) @> \tilde f^* >> H^2(\tilde X'_0)\\
@V \PD VV @V \PD VV\\
H_4\left(S^3\times S^3,\;\bigsqcup\,\Delta_k\right)  @> \tilde h_* >> H_4\left(S^3\times S^3,\;\bigsqcup\,\Delta_k\right),
\end{CD}
\]

\medskip\noindent
where PD is the Poincar\'e duality isomorphism and $h$ is the inverse of $f$. Explicitly,
\medskip
\[
\begin{split}
\tilde h (A_k) & = \tilde h (\{((z_1,z_2),(\zeta^s z_1,\zeta^{qs} z_2))\in S^3\times S^3\; |\; s\in [k-1,k]\}) \\
 & =\;\; (\{((z_2, z_1),(\zeta^{qs} z_2,\zeta^{s} z_1))\in S^3\times S^3\; |\; s\in [k-1,k]\})\\
 &=\;\; (\{((z_2, z_1),(\zeta^t z_2,\zeta^{q't} z_1))\;\in\, S^3\times S^3\; |\; t \in [q(k-1),qk]\})\\
&=\;\; A_{q(k-1)+1}+...+A_{qk}.
\end{split}
\]

\smallskip\noindent
Using the identification of $a_{k+1}$ with $t^k$, the above formula becomes $\tilde f^*(t^k) = t^{qk} + ... + t^{q(k+1)-1}$ so in particular $\tilde f^*(1) = 1 + t + \ldots + t^{q-1}$. We used our computer program to double check this answer for several prime $p$ and several $q$ and $q'$ between $0 $ and $p/2$. In every example, the computer confirmed that $\tilde f^*(1) = 1 + t + \ldots + t^{q-1}$ is a solution of \eqref{E:massey}, and in fact a unique solution up to multiplication by a power of $t$.
\end{example}

\begin{example}
Lens spaces $L(11,2)$ and $L(11,3)$ provide the smallest example of non-homeomorphic lens spaces which are homotopy equivalent and whose deleted squares have universal covers with non-vanishing Massey products, see Theorem \ref{T:miller} and also Tables 3 and 4 in Miller \cite{miller} (there is a typo in Table 4: the $(4,0)$ entry should start at 4 and not 5). The computer found no non-zero solutions of \eqref{E:massey}, implying that the deleted squares of $L(11,2)$ and $L(11,3)$ are not homotopy equivalent. Paolo Salvatore has informed us that he obtained this result in 2010 using Miller's formulas.
 \end{example}
 
 \begin{example}
We have used our methods to check that multiple pairs of lens spaces which are homotopy equivalent but not homeomorphic have non-homotopy equivalent deleted squares. Among these pairs are $L(11,2)$ and $L(11,4)$, $L(13,2)$ and $L(13,5)$, $L(13,5)$ and $L(13,6)$, $L(17,3)$ and $L(17,5)$. All of these calculations weigh in for a positive answer to the question posed in the introduction, whether the homotopy type of deleted squares distinguishes lens spaces up to homeomorphism.
\end{example}



\smallskip

\begin{thebibliography}{10} 

\bibitem{auckly}
D.~Auckly, {\em Topological methods to compute Chern--Simons invariants}, Math. Proc. Cambridge Phil. Soc. \textbf{115} (1994), 229--251

\bibitem{APS:II}
M.~Atiyah, V.~Patodi, I.~Singer,
{\em Spectral asymmetry and Riemannian geometry. II}, Mat. Proc. Cambridge Phil. Soc. \textbf{78} (1975), 405--432

\bibitem{CS}
J.~Cheeger, J.~Simons, {\em Differential characters and geometric invariants}, Lecture Notes in Math. 
\textbf{1167}, Springer, 1985

\bibitem{gordon}
C.~McA.~Gordon {\em On the G-signature theorem in dimension four}. In:  A la Recherche de la Topologie Perdue, ed. Guillou and Marin, Birkh{\"a}user, 1986

\bibitem{hatcher}
A.~Hatcher, Algebraic Topology. Cambridge Univ. Press, 2002.

\bibitem{KK}
P.~Kirk, E.~Klassen, {\em Chen--Simons invariants of 3-manifolds and representation spaces of knot groups}, Math. Ann. \textbf{287} (1990), 343--367

\bibitem{KS}
S.~Kwasik, R. Schultz, {\em All $\Z_p$ lens spaces have diffeomorphic squares}, Topology \textbf{41} (2002), 321--340

\bibitem{SL}
R. Longoni, P. Salvatore, {\em Configuration spaces are not homotopy invariant}, Topology \textbf{44} (2005), 375-380

\bibitem{miller}
M. Miller, {\em Rational homotopy models for two-point configuration spaces	of lens spaces}, Homology, Homotopy and Applications \textbf{13} (2011), 43--62.

\bibitem{rudyak}
Y.~Rudyak, On Thom spectra, orientability, and cobordism. Springer Verlag, 1998

\end{thebibliography}
\end{document}